\documentclass[reqno]{amsart}
\usepackage{setspace,amssymb}
\usepackage{ifpdf}
\ifpdf
 \usepackage[hyperindex,pagebackref]{hyperref}%
\else
 \expandafter\ifx\csname dvipdfm\endcsname\relax
 \usepackage[hypertex,hyperindex,pagebackref]{hyperref}
 \else
 \usepackage[dvipdfm,hyperindex,pagebackref]{hyperref}
 \fi
\fi
\allowdisplaybreaks[4]
\numberwithin{equation}{section}
\theoremstyle{plain}
\newtheorem{thm}{Theorem}[section]
\newtheorem{cor}{Corollary}[section]

\theoremstyle{remark}
\newtheorem{rem}{Remark}[section]
\DeclareMathOperator{\td}{d}
\newcommand{\bell}{\textup{B}}

\begin{document}

\title[Expressions for Bell polynomials and derivatives]
{Explicit expressions for a family of Bell polynomials and derivatives of some functions}

\author[F. Qi]{Feng Qi}
\address[Qi]{Department of Mathematics, College of Science, Tianjin Polytechnic University, Tianjin City, 300387, China}
\email{\href{mailto: F. Qi <qifeng618@gmail.com>}{qifeng618@gmail.com}, \href{mailto: F. Qi <qifeng618@hotmail.com>}{qifeng618@hotmail.com}, \href{mailto: F. Qi <qifeng618@qq.com>}{qifeng618@qq.com}}
\urladdr{\url{http://qifeng618.wordpress.com}}

\author[M.-M. Zheng]{Miao-Miao Zheng}
\address[Zheng]{Department of Mathematics, College of Science, Tianjin Polytechnic University, Tianjin City, 300387, China}
\email{\href{mailto: M.-M. Zheng <zhengmm0715@gmail.com>}{zhengmm0715@gmail.com}, \href{mailto: M.-M. Zheng <zhengmm0715@outlook.com>}{zhengmm0715@outlook.com}, \href{mailto:M.-M. Zheng <miao.miao.zheng@qq.com>}{miao.miao.zheng@qq.com}}

\begin{abstract}
In the paper, the authors first inductively establish explicit formulas for derivatives of the arc sine function, then derive from these explicit formulas explicit expressions for a family of Bell polynomials related to the square function, and finally apply these explicit expressions to find explicit formulas for derivatives of some elementary functions.
\end{abstract}



\thanks{This paper was typeset using \AmS-\LaTeX}

\maketitle
\section{Introduction}
Throughout this paper, we denote the set of all positive integers by $\mathbb{N}$.
\par
It is general knowledge that the $n$-th derivatives of the sine and cosine functions for $n\in\mathbb{N}$ are
\begin{equation}\label{sin-cos-n-deriv}
\sin^{(n)}x=\sin\Bigl(x+\frac{\pi}{2}n\Bigr)\quad \text{and}\quad
\cos^{(n)}x=\cos\Bigl(x+\frac{\pi}{2}n\Bigr).
\end{equation}
\par
In~\cite{Tan-Cot-Bernulli-No.tex, derivative-tan-cot.tex}, among other things, the following explicit formulas for the $n$-th derivatives of the tangent and cotangent functions were inductively established:
\begin{multline}\label{tan-deriv-unity}
\tan^{(n)}x=\frac1{\cos^{n+1}x}\Biggl\{\frac12\alpha_{n,\frac{1+(-1)^n}2} \sin\biggl[\frac{1+(-1)^n}2x +\frac{1-(-1)^n}2\frac\pi2\biggr]\\
+\sum_{i=1}^{\frac12[n-1-\frac{1+(-1)^n}2]} \alpha_{n,2i+\frac{1+(-1)^n}2}\sin\biggl[\biggl(2i+\frac{1+(-1)^n}2\biggr)x +\frac{1-(-1)^n}2\frac\pi2\biggr]\Biggr\}
\end{multline}
and
\begin{multline}\label{cotan-deriv-n}
\cot^{(n)}x=\frac1{\sin^{n+1}x}\Biggl\{\frac12\beta_{n,\frac{1+(-1)^n}2} \cos\biggl[\frac{1+(-1)^n}2x\biggr]\\
+\sum_{i=1}^{\frac12[n-1-\frac{1+(-1)^n}2]} \beta_{n,2i+\frac{1+(-1)^n}2} \cos\biggl[\biggl(2i+\frac{1+(-1)^n}2\biggr)x\biggr]\Biggr\},
\end{multline}
where
\begin{equation}\label{a(p-q)-unify}
\alpha_{p,q}=(-1)^{\frac12[q-\frac{1+(-1)^p}2]}[1-(-1)^{p-q}]\sum_{\ell=0}^{\frac{p-q-1}2} (-1)^{\ell}\binom{p+1}{\ell}\biggl(\frac{p-q-1}2-\ell+1\biggr)^p
\end{equation}
and
\begin{equation}
\beta_{p,q}=(-1)^{\frac{1-(-1)^p}2}[1-(-1)^{p-q}]\sum_{\ell=0}^{\frac{p-q-1}2} (-1)^{\ell}\binom{p+1}{\ell}\biggl(\frac{p-q-1}2-\ell+1\biggr)^p
\end{equation}
for $p>q\geq0$. These formulas have been applied in~\cite{Tan-Cot-Bernulli-No.tex, singularity-combined.tex, polygamma-sigularity.tex}.
\par
In~\cite[Theorem~2]{exp-reciprocal-cm.tex} and its formally published paper~\cite[Theorem~2.2]{exp-reciprocal-cm-IJOPCM.tex}, the following explicit formula for the $n$-th derivative of the exponential function $e^{\pm1/t}$ was inductively obtained:
\begin{equation}\label{exp-frac1x-expans}
\bigl(e^{\pm1/t}\bigr)^{(n)}=(-1)^n\frac{e^{\pm1/t}}{t^{2n}} \sum_{k=0}^{n-1}(\pm1)^{n-k}L(n,n-k)t^{k},
\end{equation}
where
\begin{equation}\label{a-i-k-dfn}
L(n,k)=\binom{n-1}{k-1}\frac{n!}{k!}
\end{equation}
are called Lah numbers in combinatorics. By the way, Lah number $L(n,k)$ were discovered by Ivo Lah in 1955 and it counts the number of ways a set of $n$ elements can be partitioned into $k$ nonempty linearly ordered subsets. The formula~\eqref{exp-frac1x-expans} was also recovered in~\cite{DMST-MM-2013-Exp} and have been applied in~\cite{Lah-Num-Int-Prop-Corr.tex, notes-Stirl-No-JNT-rev.tex, Bell-Stirling-Lah-simp.tex, Lah-No-Identity.tex, Filomat-36-73-1.tex, Bessel-ineq-Dgree-CM.tex, QiBerg.tex, exp-psi-cm-revised.tex} respectively.
\par
In combinatorics, Bell polynomials of the second kind, or say, the partial Bell polynomials, denoted by $\bell_{n,k}(x_1,x_2,\dotsc,x_{n-k+1})$ for $n\ge k\ge1$, are defined by
\begin{equation}\label{bell-polyn-dfn}
\bell_{n,k}(x_1,x_2,\dotsc,x_{n-k+1})=\sum_{\substack{1\le i\le n,\ell_i\in\mathbb{N}\\ \sum_{i=1}^ni\ell_i=n\\ \sum_{i=1}^n\ell_i=k}}\frac{n!}{\prod_{i=1}^{n-k+1}\ell_i!} \prod_{i=1}^{n-k+1}\Bigl(\frac{x_i}{i!}\Bigr)^{\ell_i}.
\end{equation}
See~\cite[p.~134, Theorem~A]{Comtet-Combinatorics-74}.
The famous Fa\`a di Bruno formula may be described in terms of Bell polynomials of the second kind $\bell_{n,k}(x_1,x_2,\dotsc,x_{n-k+1})$ by
\begin{equation}\label{Bruno-Bell-Polynomial}
\frac{\td^n}{\td t^n}f\circ h(t)
=\sum_{k=1}^nf^{(k)}(h(t)) \bell_{n,k}\bigl(h'(t),h''(t),\dotsc,h^{(n-k+1)}(t)\bigr).
\end{equation}
See~\cite[p.~139, Theorem~C]{Comtet-Combinatorics-74}. This is an effective tool to compute the $n$-th derivatives of some composite functions. However, generally it is not an easy matter to explicitly find Bell polynomials $\bell_{n,k}$.
\par
In this paper, motivated by inductive deductions and extensive applications of the formulas~\eqref{tan-deriv-unity}, \eqref{cotan-deriv-n}, and~\eqref{exp-frac1x-expans}, we first inductively establish explicit formulas for the $n$-th derivatives of the functions $\arcsin x$ and $\arccos x$, then derive from these explicit formulas explicit expressions of Bell polynomials $\bell_{n,k}(x,1,\overbrace{0,\dotsc,0}^{n-k-1})$, and finally apply these explicit expressions to compute the $n$-th derivatives of some elementary funcitons involving the square function $x^2$.

\section{Explicit formulas for derivatives of $\arcsin x$ and $\arccos x$}

In this section, we will inductively establish explicit formulas for the $n$-th derivatives of the functions $\arcsin x$ and $\arccos x$. Essentially, we will find explicit formulas for the $n$-th derivatives of the function $\frac1{\sqrt{1-x^2}\,}$.

\begin{thm}\label{arcsine}
For $k\in\mathbb{N}$ and $x\in(-1,1)$, the $n$-th derivatives of the functions $\arcsin x$ and $\arccos x$ may be computed by
\begin{equation}\label{odd-n-deriv-arcsine}
\arcsin^{(2k-1)}x=-\arccos^{(2k-1)}x=\sum_{i=0}^{k-1}a_{2k-1,2i}\frac{x^{2i}}{(1-x^2)^{k+i-1/2}}
\end{equation}
and
\begin{equation}\label{even-n-deriv-arcsine}
\arcsin^{(2k)}x=-\arccos^{(2k)}x=\sum_{i=0}^{k-1}a_{2k,2i+1}\frac{x^{2i+1}}{(1-x^2)^{k+i+1/2}},
\end{equation}
where
\begin{align}\label{recur-1}
a_{2k-1,0}&={[(2k-3)!!]}^{2},\\
\label{recur-5}
a_{2k,1}&={[(2k-1)!!]}^{2},\\
\label{arcsin-driv-coeffi-1}
a_{k+1,k}&=(2k-1)!!,
\end{align}
and
\begin{equation}
\label{arcsin-driv-coeffi-n}
a_{m,k}=\frac{(m+k-2)!!(m-1)!}{2^{m-k-2}k!}
\end{equation}
for $m\ge k+2\ge3$.
\end{thm}

\begin{proof}
It is easy to obtain that
\begin{equation*}
(\arcsin x)'=\frac1{(1-x^2)^{1/2}}\quad \text{and}\quad
(\arcsin{x})''=\frac{x}{(1-x^2)^{3/2}}.
\end{equation*}
This means the special case $k=1$ in~\eqref{recur-1} and~\eqref{arcsin-driv-coeffi-1}.
Therefore, the formulas~\eqref{odd-n-deriv-arcsine} and~\eqref{even-n-deriv-arcsine} are valid for $k=1$.
\par
Assume that the formulas~\eqref{odd-n-deriv-arcsine} and~\eqref{even-n-deriv-arcsine} are valid for $k>1$. By this inductive hypothesis and a direct differentiation, we have
\begin{multline*}
\arcsin^{(2k)}x=\bigl[\arcsin^{(2k-1)}x\bigr]'
=\Biggl[\sum_{i=0}^{k-1}a_{2k-1,2i}\frac{x^{2i}}{(1-x^2)^{k+i-1/2}}\Biggr]'\\
\begin{aligned}
&=\sum_{i=0}^{k-1}a_{2k-1,2i}\Biggl(\frac{(x^{2i})'(1-x^2)^{k+i-1/2}-x^{2i}[(1-x^2)^{k+i-1/2}]'} {(1-x^2)^{2(k+i-1/2)}}\Biggr)\\
&=\sum_{i=0}^{k-1}a_{2k-1,2i}\biggl[\frac{2ix^{2i-1}(1-x^2)^{k+i-1/2} +2(k+i-1/2)x^{2i+1}(1-x^2)^{k+i-3/2}} {(1-x^2)^{2(k+i-1/2)}}\biggr]\\
&=\sum_{i=0}^{k-1}\biggl[a_{2k-1,2i}\frac{2ix^{2i-1}}{(1-x^2)^{k+i-1/2}}+a_{2k-1,2i}\frac{
2(k+i-1/2)x^{2i+1}}{(1-x^2)^{k+i+1/2}}\biggr]\\
&=\sum_{i=0}^{k-2}a_{2k-1,2(i+1)}\frac{2(i+1)x^{2i+1}}{(1-x^2)^{k+i+1/2}}
+\sum_{i=0}^{k-1}a_{2k-1,2i}\frac{2(k+i-1/2)x^{2i+1}}{(1-x^2)^{k+i+1/2}}\\
&=\sum_{i=0}^{k-2}\biggl[2(i+1)a_{2k-1,2(i+1)} +2\biggl(k+i-\frac12\biggr)a_{2k-1,2i}\biggr] \frac{x^{2i+1}}{(1-x^2)^{k+i+1/2}}\\
&\quad +2\biggl(2k-\frac32\biggr)a_{2k-1,2k-2}\frac{x^{2k-1}}{(1-x^2)^{2k-1/2}}
\end{aligned}
\end{multline*}
and
\begin{multline*}
\arcsin^{(2k+1)}x=\bigl[\arcsin^{(2k)}x\bigr]'
=\Biggl[\sum_{i=0}^{k-1}a_{2k,2i+1}\frac{x^{2i+1}}{(1-x^2)^{k+i+1/2}}\Biggr]'\\
\begin{aligned}
&=\sum_{i=0}^{k-1}a_{2k,2i+1}\biggl(\frac{(x^{2i+1})'(1-x^2)^{k+i+1/2}-x^{2i+1}[(1-x^2)^{k+i+1/2}]'} {(1-x^2)^{2(k+i+1/2)}}\biggr)\\
&=\sum_{i=0}^{k-1}a_{2k,2i+1}\biggl[\frac{(2i+1)x^{2i}(1-x^2)^{k+i+1/2} +2(k+i+1/2)x^{2i+2}(1-x^2)^{k+i-1/2}} {(1-x^2)^{2(k+i+1/2)}}\biggr]\\
&=\sum_{i=0}^{k-1}\biggl[a_{2k,2i+1}\frac{(2i+1)x^{2i}}{(1-x^2)^{k+i+1/2}}+a_{2k,2i+1}\frac{
2(k+i+1/2)x^{2i+2}}{(1-x^2)^{k+i+3/2}}\biggr]
\end{aligned}\\
\begin{aligned}
&=\sum_{i=0}^{k-1}a_{2k,2i+1}\frac{(2i+1)x^{2i}}{(1-x^2)^{k+i+1/2}} +\sum_{i=1}^{k}a_{2k,2i-1}\frac{2(k+i-1/2)x^{2i}}{(1-x^2)^{k+i+1/2}}\\
&=\sum_{i=1}^{k-1}\biggl[(2i+1)a_{2k,2i+1}+2\biggl(k+i-\frac12\biggr)a_{2k,2i-1}\biggr] \frac{x^{2i}}{(1-x^2)^{k+i+1/2}}\\
&\quad+a_{2k,1}\frac1{(1-x^2)^{k+1/2}} +2\biggl(2k-\frac12\biggr)a_{2k,2k-1}\frac{x^{2k}}{(1-x^2)^{2k+1/2}}.
\end{aligned}
\end{multline*}
Comparing the above two formulas with
\begin{equation*}
\sum_{i=0}^{k-1}a_{2k,2i+1}\frac{x^{2i+1}}{(1-x^2)^{k+i+1/2}} \quad \text{and}\quad \sum_{i=0}^{k}a_{2k+1,2i}\frac{x^{2i}}{(1-x^2)^{k+i+1/2}}
\end{equation*}
respectively yields the recursion formulas
\begin{align}\label{recursion-coefficients-a}
a_{2k,2k-1}&=(4k-3)a_{2k-1,2k-2},\\
a_{2k,2i+1}&=2(i+1)a_{2k-1,2(i+1)}+(2k+2i-1)a_{2k-1,2i}\label{recur-2}
\end{align}
for $0\le{i}<k-1$, and
\begin{align}\label{recursion-formulas-coefficients-b}
a_{2k+1,0}&=a_{2k,1} ,\\
a_{2k+1,2k}&=(4k-1)a_{2k,2k-1},\label{recur-3}\\
a_{2k+1,2i}&=(2i+1)a_{2k,2i+1}+(2k+2i-1)a_{2k,2i-1}\label{recur-4}
\end{align}
for $1\le i\le k-1$.
\par
From~\eqref{recur-1}, \eqref{recursion-coefficients-a}, and~\eqref{recur-3}, it is easy to derive that
\begin{equation*}
a_{2k+1,2k}=(4k-1)!!\quad \text{and}\quad a_{2k,2k-1}=(4k-3)!!,
\end{equation*}
which may be unified into~\eqref{arcsin-driv-coeffi-1} for $k\ge2$.
\par
From
\begin{gather*}
a_{3,0}=a_{2,1}=1=(1!!)^2,\quad
a_{5,0}=a_{4,1}=9=(3!!)^2,\\
a_{7,0}=a_{6,1}=225=(5!!)^2,\quad
a_{9,0}=a_{8,1}=11025=(7!!)^2,
\end{gather*}
it is not difficult to inductively conclude~\eqref{recur-5}.
\par
Letting $i=k-2$ and $i=k-1$ for $k\ge2$ in~\eqref{recur-2} and~\eqref{recur-4} respectively yields
\begin{equation*}
a_{2k,2k-3}=2(k-1)a_{2k-1,2k-2}+(4k-5)a_{2k-1,2k-4}
\end{equation*}
and
\begin{equation*}
a_{2k+1,2k-2}=(2k-1)a_{2k,2k-1}+(4k-3)a_{2k,2k-3}.
\end{equation*}
Combining these two recurrence formulas with~\eqref{recur-5} and~\eqref{arcsin-driv-coeffi-1} and recurring give
\begin{equation}\label{arcsin-driv-coeffi-3}
a_{k+3,k}=(2k+1)!!\sum_{\ell=1}^{k+1}\ell=(2k+1)!!\frac{(k+1)(k+2)}{2}, \quad k\ge0.
\end{equation}
\par
Taking $i=k-3$ and $i=k-2$ for $k\ge3$ in~\eqref{recur-2} and~\eqref{recur-4} respectively yields
\begin{equation*}
a_{2k,2k-5}=2(k-2)a_{2k-1,2k-4}+(4k-7)a_{2k-1,2k-6}
\end{equation*}
and
\begin{equation*}
a_{2k+1,2k-4}=(2k-3)a_{2k,2k-3}+(4k-5)a_{2k,2k-5}.
\end{equation*}
Combining these two recurrence formulas with~\eqref{recur-5} and~\eqref{arcsin-driv-coeffi-1} and recurring give
\begin{equation}\label{arcsin-driv-coeffi-5}
\begin{split}
a_{k+5,k}&=(2k+3)!!\sum_{\ell=1}^{k+1}\frac{\ell(\ell+1)(\ell+2)}{2}\\
&=(2k+3)!!\frac{(k+1) (k+2) (k+3) (k+4)}{8}
\end{split}
\end{equation}
for $k\ge0$.
\par
Similarly as above, by induction, we obtain
\begin{equation}\label{arcsin-driv-coeffi-k}
a_{k+\ell,k}=\frac{(2k+\ell-2)!!}{2^{\ell-2}}\prod_{i=1}^{\ell-1}{(k+i)}
=\frac{(2k+\ell-2)!!}{2^{\ell-2}}\frac{(k+\ell-1)!}{k!}, \quad \ell\ge2.
\end{equation}
Letting $\ell=m-k$ in~\eqref{arcsin-driv-coeffi-k} leads to~\eqref{arcsin-driv-coeffi-n}.
The proof of Theorem~\ref{arcsine} is complete.
\end{proof}

The formulas~\eqref{odd-n-deriv-arcsine} and~\eqref{even-n-deriv-arcsine} may be straightforwardly unified as the following corollary.

\begin{cor}
For $n\in\mathbb{N}$ and $x\in(-1,1)$, the $n$-th derivatives of the functions $\arcsin x$ and $\arccos x$ may be computed by
\begin{equation}
\begin{split}\label{unified-formula}
\arcsin^{(n)}x&=-\arccos^{(n)}x\\
&=\sum_{i=0}^{\frac12[n+\frac{1-(-1)^n}2]-1}a_{n,2i+\frac{1+(-1)^n}2} \frac{x^{2i+\frac{1+(-1)^n}2}}{(1-x^2)^{i+\frac12[n+\frac{1-(-1)^n}2]+\frac{(-1)^n}2}},
\end{split}
\end{equation}
where $a_{n,2i+\frac{1+(-1)^n}2}$ are defined by~\eqref{recur-1}, \eqref{recur-5}, \eqref{arcsin-driv-coeffi-1}, and~\eqref{arcsin-driv-coeffi-n}.
\end{cor}

From the formula~\eqref{unified-formula}, we may derive the $n$-th derivative of the elementary function $\frac1{\sqrt{1-x^2}\,}$ as follows.

\begin{cor}\label{cor-2.2-arcsin}
For $n\in\mathbb{N}$ and $x\in(-1,1)$, the $n$-th derivatives of the function $\frac1{\sqrt{1-x^2}\,}$ may be computed by
\begin{equation}\label{sqrt-(1-x-square)}
\biggl(\frac1{\sqrt{1-x^2}\,}\biggr)^{(n)}
=\sum_{k=0}^{\frac12[n-\frac{1-(-1)^{n}}2]}a_{n+1,2k+\frac{1-(-1)^{n}}2} \frac{x^{2k+\frac{1-(-1)^{n}}2}}{(1-x^2)^{k+\frac12[n+\frac{1+(-1)^{n}}2]+\frac{1-(-1)^{n}}2}},
\end{equation}
where $a_{n,2k+\frac{1+(-1)^n}2}$ are defined by~\eqref{recur-1}, \eqref{recur-5}, \eqref{arcsin-driv-coeffi-1}, and~\eqref{arcsin-driv-coeffi-n}.
\end{cor}

\section{Explicit expressions of Bell polynomials}

In this section, by virtue of Corollary~\ref{cor-2.2-arcsin}, we will derive explicit expressions of Bell polynomials $\bell_{n,k}(x,1,\overbrace{0,\dotsc,0}^{n-k-1})$.

\begin{thm}\label{Bell-exp-thm}
For $n\in\mathbb{N}$, Bell polynomials $\bell_{n,k}$ satisfy
\begin{gather}\label{k=frac-n2}
\bell_{2n-1,n-1}(x,1,\overbrace{0,\dotsc,0}^{n-1})=0,\\
\label{bell-no-0-1-0-one}
\bell_{2n,n}(x,1,\overbrace{0,\dotsc,0}^{n-1})=(2n-1)!!,\\
\label{Bell-eq-1}
\bell_{n,k}(x,1,\overbrace{0,\dotsc,0}^{n-k-1})=0, \quad 1\le k<\frac12\biggl[n-\frac{1-(-1)^n}2\biggr],\\
\label{Bell-eq-unified}
\bell_{n,k}(x,1,\overbrace{0,\dotsc,0}^{n-k-1})
=\frac{a_{n+1,2k-n}}{(2k-1)!!}{x^{2k-n}}, \quad n\ge k>\frac12\biggl[n-\frac{1-(-1)^n}2\biggr],
\end{gather}
where $a_{n,k}$ are defined by~\eqref{recur-1}, \eqref{recur-5}, \eqref{arcsin-driv-coeffi-1}, and~\eqref{arcsin-driv-coeffi-n}.
\end{thm}

\begin{proof}
Let $v=v(x)=1-x^2$. Then, by Fa\'a di Bruno formula~\eqref{Bruno-Bell-Polynomial}, we have
\begin{align*}
\frac{\td^n}{\td x^n}\biggl(\frac1{\sqrt{1-x^2}\,}\biggr)
&=\sum_{k=1}^n\biggl(\frac1{\sqrt{v}\,}\biggr)^{(k)} \bell_{n,k}\bigl(v'(x),v''(x),\dotsc,v^{(n-k+1)}(x)\bigr)\\
&=\sum_{k=1}^n(-1)^k \prod_{\ell=0}^{k-1}\biggl(\frac12+\ell\biggr) \frac1{v^{k+1/2}} \bell_{n,k}(-2x,-2,\overbrace{0,\dotsc,0}^{n-k-1}).
\end{align*}
By the formula
\begin{equation}\label{Bell(n-n)}
\bell_{n,k}\bigl(abx_1,ab^2x_2,\dotsc,ab^{n-k+1}x_{n-k+1}\bigr)
=a^kb^n\bell_{n,k}(x_1,x_2,\dotsc,x_{n-k+1})
\end{equation}
in~\cite[p.~135]{Comtet-Combinatorics-74}, we have
\begin{equation*}
\bell_{n,k}(-2x,-2,\overbrace{0,\dotsc,0}^{n-k-1})
=(-2)^k\bell_{n,k}(x,1,\overbrace{0,\dotsc,0}^{n-k-1}).
\end{equation*}
Therefore, we have
\begin{align*}
\frac{\td^n}{\td x^n}\biggl(\frac1{\sqrt{1-x^2}\,}\biggr)
&=\sum_{k=1}^n \prod_{\ell=0}^{k-1}(2\ell+1) \frac1{v^{k+1/2}} \bell_{n,k}(x,1,\overbrace{0,\dotsc,0}^{n-k-1})\\
&=\sum_{k=1}^n\frac{(2k-1)!!}{(1-x^2)^{k+1/2}} \bell_{n,k}(x,1,\overbrace{0,\dotsc,0}^{n-k-1}).
\end{align*}
Comparing this with the formula~\eqref{sqrt-(1-x-square)} reveals that
\begin{equation}\label{bell-polyn-one}
\sum_{k=1}^{2n}\frac{(2k-1)!!}{(1-x^2)^{k+1/2}} \bell_{2n,k}(x,1,\overbrace{0,\dotsc,0}^{2n-k-1})
=\frac1{(1-x^2)^n}\sum_{k=0}^{n}a_{2n+1,2k} \frac{x^{2k}}{(1-x^2)^{k+1/2}}
\end{equation}
and
\begin{equation}\label{bell-polyn-two}
\sum_{k=1}^{2n-1}\frac{(2k-1)!!}{(1-x^2)^{k+1/2}} \bell_{2n-1,k}(x,1,\overbrace{0,\dotsc,0}^{2n-k-2})
=\frac1{(1-x^2)^n}\sum_{k=0}^{n-1}a_{2n,2k+1} \frac{x^{2k+1}}{(1-x^2)^{k+1/2}}
\end{equation}
for $n\in\mathbb{N}$.
Multiplying on both sides of~\eqref{bell-polyn-one} and~\eqref{bell-polyn-two} by $(1-x^2)^{2n+1/2}$ gives
\begin{equation}\label{bell-vanish-one}
\sum_{k=1}^{2n}{(2k-1)!!}(1-x^2)^{2n-k} \bell_{2n,k}(x,1,\overbrace{0,\dotsc,0}^{2n-k-1})
=\sum_{k=0}^{n}a_{2n+1,2k}{x^{2k}}{(1-x^2)^{n-k}}
\end{equation}
and
\begin{equation}\label{bell-vanish-two}
\sum_{k=1}^{2n-1}{(2k-1)!!}(1-x^2)^{2n-k} \bell_{2n-1,k}(x,1,\overbrace{0,\dotsc,0}^{2n-k-2})
=\sum_{k=0}^{n-1}a_{2n,2k+1}{x^{2k+1}}{(1-x^2)^{n-k}}
\end{equation}
for $n\in\mathbb{N}$. Equating these two equations finds that
\begin{enumerate}
\item
when $n>k$, Bell polynomials $\bell_{2n,k}(x,1,\overbrace{0,\dotsc,0}^{2n-k-1})=0$;
\item
when $n>k+1$, Bell polynomials $\bell_{2n-1,k}(x,1,\overbrace{0,\dotsc,0}^{2n-k-2})=0$.
\end{enumerate}
These two results may be unified as the formula~\eqref{Bell-eq-1}.
\par
Making use of the formula~\eqref{Bell-eq-1}, the formulas~\eqref{bell-vanish-one} and~\eqref{bell-vanish-two} are reduced to
\begin{align*}
&\quad\sum_{k=n}^{2n}{(2k-1)!!}(1-x^2)^{2n-k} \bell_{2n,k}(x,1,\overbrace{0,\dotsc,0}^{2n-k-1})\\
&=\sum_{\ell=0}^{n}{(2n+2\ell-1)!!}(1-x^2)^{n-\ell} \bell_{2n,n+\ell}(x,1,\overbrace{0,\dotsc,0}^{n-\ell-1})\\
&=\sum_{k=0}^{n}a_{2n+1,2k}{x^{2k}}{(1-x^2)^{n-k}}
\end{align*}
and
\begin{align*}
&\quad\sum_{k=n-1}^{2n-1}{(2k-1)!!}(1-x^2)^{2n-k} \bell_{2n-1,k}(x,1,\overbrace{0,\dotsc,0}^{2n-k-2})\\
&={(2n-3)!!}(1-x^2)^{n+1} \bell_{2n-1,n-1}(x,1,\overbrace{0,\dotsc,0}^{n-1})\\\
&\quad +\sum_{\ell=0}^{n-1}{(2n+2\ell-1)!!}(1-x^2)^{n-\ell} \bell_{2n-1,n+\ell}(x,1,\overbrace{0,\dotsc,0}^{n-\ell-2})\\
&=\sum_{k=0}^{n-1}a_{2n,2k+1}{x^{2k+1}}{(1-x^2)^{n-k}}
\end{align*}
for $n\in\mathbb{N}$. Equating the above equations figures out the formula~\eqref{k=frac-n2},
\begin{equation}
\bell_{2n,n+k}(x,1,\overbrace{0,\dotsc,0}^{n-k-1})
=\frac{a_{2n+1,2k}}{(2n+2k-1)!!}{x^{2k}},\quad 0\le k\le n, \label{Bell-eq-2}
\end{equation}
and
\begin{equation}
\bell_{2n-1,n+k}(x,1,\overbrace{0,\dotsc,0}^{n-k-2})
=\frac{a_{2n,2k+1}}{(2n+2k-1)!!}{x^{2k+1}}, \quad 0\le k\le n-1.\label{Bell-eq-3}
\end{equation}
The formulas~\eqref{Bell-eq-2} and~\eqref{Bell-eq-3} may be reformulated as~\eqref{bell-no-0-1-0-one} and~\eqref{Bell-eq-unified}.
The proof of Theorem~\ref{Bell-exp-thm} is complete.
\end{proof}

\section{Explicit formulas for the $n$-th derivatives of some functions}

In this section, with the help of Theorem~\ref{Bell-exp-thm}, we will discover explicit formulas for the $n$-th derivatives of some elementary functions.

\begin{thm}\label{arctan-deriv-thm}
For $\ell\in\mathbb{N}$, we have
\begin{equation}
(\arctan x)^{(2\ell)}
=\frac{(-1)^\ell x}{(1+x^2)^{\ell+1}}\sum_{k=0}^{\ell-1}(-1)^{k} \frac{[2(k+\ell)]!!}{[2(k+\ell)-1]!!} a_{2\ell,2k+1}\biggl(\frac{x^{2}}{1+x^2}\biggr)^k
\end{equation}
and
\begin{equation}
(\arctan x)^{(2\ell-1)}
=\frac{(-1)^{\ell-1}}{(1+x^2)^\ell}\sum_{k=0}^{\ell-1}(-1)^k \frac{[2(k+\ell-1)]!!} {[2(k+\ell-1)-1]!!} a_{2\ell-1,2k}\biggl(\frac{x^{2}}{1+x^2}\biggr)^k,
\end{equation}
where $a_{n,k}$ are defined by~\eqref{recur-1}, \eqref{recur-5}, \eqref{arcsin-driv-coeffi-1}, and~\eqref{arcsin-driv-coeffi-n}.
\end{thm}

\begin{proof}
Let $v=v(x)=1+x^2$. Then, by Fa\'a di Bruno formula~\eqref{Bruno-Bell-Polynomial} and the formula~\eqref{Bell(n-n)}, we obtain
\begin{align*}
(\arctan x)^{(n)}&=\biggl(\frac1{1+x^2}\biggr)^{(n-1)}\\
&=\sum_{k=1}^{n-1}\biggl(\frac1{v}\biggr)^{(k)}\bell_{n-1,k}(2x,2,\overbrace{0,\dotsc,0}^{n-k-2})\\
&=\sum_{k=1}^{n-1}(-1)^k\frac{k!}{v^{k+1}}2^k\bell_{n-1,k}(x,1,\overbrace{0,\dotsc,0}^{n-k-2})\\
&=\sum_{k=1}^{n-1}(-1)^k\frac{(2k)!!}{(1+x^2)^{k+1}}\bell_{n-1,k}(x,1,\overbrace{0,\dotsc,0}^{n-k-2}).
\end{align*}
Hence, by Theorem~\ref{Bell-exp-thm}, it follows that
\begin{enumerate}
\item
when $n=2\ell$, we have
\begin{align*}
(\arctan x)^{(2\ell)}&=\biggl(\frac1{1+x^2}\biggr)^{(2\ell-1)}\\
&=\sum_{k=1}^{2\ell-1}(-1)^k\frac{(2k)!!}{(1+x^2)^{k+1}} \bell_{2\ell-1,k}(x,1,\overbrace{0,\dotsc,0}^{2\ell-k-2})\\
&=\sum_{k=\ell}^{2\ell-1}(-1)^k \frac{(2k)!!a_{2\ell,2(k-\ell)+1}}{(2k-1)!!} \frac{x^{2k-2\ell+1}}{(1+x^2)^{k+1}};
\end{align*}
\item
when $n=2\ell-1$, we have
\begin{align*}
(\arctan x)^{(2\ell-1)}&=\biggl(\frac1{1+x^2}\biggr)^{(2\ell-2)}\\
&=\sum_{k=1}^{2\ell-2}(-1)^k\frac{(2k)!!}{(1+x^2)^{k+1}} \bell_{2\ell-2,k}(2x,2,\overbrace{0,\dotsc,0}^{2\ell-k-3})\\
&=\sum_{k=\ell-1}^{2\ell-2}(-1)^k \frac{(2k)!!a_{2\ell-1,2(k-\ell+1)}}{(2k-1)!!} \frac{x^{2(k-\ell+1)}}{(1+x^2)^{k+1}}.
\end{align*}
\end{enumerate}
The proof of Theorem~\ref{arctan-deriv-thm} is complete.
\end{proof}

\begin{rem}
After this paper was completed on 20 March 2014, the authors searched out on 27 March 2014 the papers~\cite{Adegoke-Layeni-AMENotes-2010, Lampret-AMEnotes-2011} in which several formulas for the $n$-th derivatives of the inverse tangent function were established and discussed.
\end{rem}

\begin{thm}\label{exp=x-square-thm}
For $\ell\in\mathbb{N}$, we have
\begin{equation}
\frac{\td^{2\ell}e^{\pm x^2}}{\td x^{2\ell}}
=(\pm2)^{\ell}e^{\pm x^2}\sum_{k=0}^{\ell}\frac{(\pm2)^k}{[2(k+\ell)-1]!!}a_{2\ell+1,2k} {x^{2k}}
\end{equation}
and
\begin{equation}
\frac{\td^{2\ell-1}e^{\pm x^2}}{\td x^{2\ell-1}}
=(\pm2)^{\ell}xe^{\pm x^2}\sum_{k=0}^{\ell-1} \frac{(\pm2)^k}{[2(k+\ell)-1]!!} a_{2\ell,2k+1}{x^{2k}},
\end{equation}
where $a_{n,k}$ are defined by~\eqref{recur-1}, \eqref{recur-5}, \eqref{arcsin-driv-coeffi-1}, and~\eqref{arcsin-driv-coeffi-n}.
\end{thm}

\begin{proof}
Let $u=u(x)=x^2$. Then, by Fa\'a di Bruno formula~\eqref{Bruno-Bell-Polynomial} and the formula~\eqref{Bell(n-n)}, we acquire
\begin{equation*}
\frac{\td^ne^{\pm x^2}}{\td x^n}
=\sum_{k=1}^n\frac{\td^ke^{\pm u}}{\td u^k} \bell_{n,k}(2x,2,\overbrace{0,\dotsc,0}^{n-k-1})
=e^{\pm x^2}\sum_{k=1}^n(\pm2)^k\bell_{n,k}(x,1,\overbrace{0,\dotsc,0}^{n-k-1}).
\end{equation*}
Hence, by Theorem~\ref{Bell-exp-thm}, it follows that
\begin{enumerate}
\item
when $n=2\ell$, we have
\begin{align*}
\frac{\td^{2\ell}e^{\pm x^2}}{\td x^{2\ell}}
&=e^{\pm x^2}\sum_{k=1}^{2\ell}(\pm2)^k \bell_{2\ell,k}(x,1,\overbrace{0,\dotsc,0}^{2\ell-k-1})\\
&=e^{\pm x^2}\sum_{k=\ell}^{2\ell}(\pm2)^k \frac{a_{2\ell+1,2(k-\ell)}}{(2k-1)!!}{x^{2(k-\ell)}};
\end{align*}
\item
when $n=2\ell-1$, we have
\begin{align*}
\frac{\td^{2\ell-1}e^{\pm x^2}}{\td x^{2\ell-1}}
&=e^{\pm x^2}\sum_{k=1}^{2\ell-1}(\pm2)^k \bell_{{2\ell-1},k}(x,1,\overbrace{0,\dotsc,0}^{2\ell-k-2})\\
&=e^{\pm x^2}\sum_{k=\ell}^{2\ell-1}(\pm2)^k\frac{a_{2\ell,2(k-\ell)+1}}{(2k-1)!!}{x^{2(k-\ell)+1}}.
\end{align*}
\end{enumerate}
The proof of Theorem~\ref{exp=x-square-thm} is complete.
\end{proof}

\begin{thm}\label{sin-cos-deriv-thm}
For $\ell\in\mathbb{N}$, we have
\begin{align}
\frac{\td^{2\ell}\sin(x^2)}{\td x^{2\ell}}
&=2^\ell\sum_{k=0}^{\ell} \frac{2^ka_{2\ell+1,2k}}{[2(k+\ell)-1]!!} {x^{2k}} \sin\Bigl(x^2+\frac{\pi}{2}(k+\ell)\Bigr),\\
\frac{\td^{2\ell-1}\sin(x^2)}{\td x^{2\ell-1}}
&=2^\ell\sum_{k=0}^{\ell-1} \frac{2^ka_{2\ell,2k+1}}{[2(k+\ell)-1]!!} {x^{2k+1}} \sin\Bigl(x^2+\frac{\pi}{2}(k+\ell)\Bigr),\\
\frac{\td^{2\ell}\cos(x^2)}{\td x^{2\ell}}
&=2^\ell\sum_{k=0}^{\ell} \frac{2^ka_{2\ell+1,2k}}{[2(k+\ell)-1]!!} {x^{2k}} \cos\Bigl(x^2+\frac{\pi}{2}(k+\ell)\Bigr),\\
\frac{\td^{2\ell-1}\cos(x^2)}{\td x^{2\ell-1}}
&=2^\ell\sum_{k=0}^{\ell-1} \frac{2^ka_{2\ell,2k+1}}{[2(k+\ell)-1]!!} {x^{2k+1}} \cos\Bigl(x^2+\frac{\pi}{2}(k+\ell)\Bigr),
\end{align}
where $a_{n,k}$ are defined by~\eqref{recur-1}, \eqref{recur-5}, \eqref{arcsin-driv-coeffi-1}, and~\eqref{arcsin-driv-coeffi-n}.
\end{thm}

\begin{proof}
Let $u=u(x)=x^2$. Then, by Fa\'a di Bruno formula~\eqref{Bruno-Bell-Polynomial} and the formulas~\eqref{Bell(n-n)} and~\eqref{sin-cos-n-deriv}, we gain
\begin{align*}
\frac{\td^n\sin(x^2)}{\td x^n}&=\sum_{k=1}^n\frac{\td^k\sin u}{\td u^k} \bell_{n,k}(2x,2,\overbrace{0,\dotsc,0}^{n-k-1})\\
&=\sum_{k=1}^n\sin\Bigl(x^2+\frac{\pi}{2}k\Bigr) 2^k\bell_{n,k}(x,1,\overbrace{0,\dotsc,0}^{n-k-1}).
\end{align*}
Accordingly, by Theorem~\ref{Bell-exp-thm}, it follows that
\begin{enumerate}
\item
when $n=2\ell$, we have
\begin{align*}
\frac{\td^{2\ell}\sin(x^2)}{\td x^{2\ell}}
&=\sum_{k=1}^{2\ell}\sin\Bigl(x^2+\frac{\pi}{2}k\Bigr) 2^k\bell_{2\ell,k}(x,1,\overbrace{0,\dotsc,0}^{2\ell-k-1})\\
&=\sum_{k=\ell}^{2\ell}2^k\sin\Bigl(x^2+\frac{\pi}{2}k\Bigr) \frac{a_{2\ell+1,2(k-\ell)}}{(2k-1)!!}{x^{2(k-\ell)}};
\end{align*}
\item
when $n=2\ell-1$, we have
\begin{align*}
\frac{\td^{2\ell-1}\sin(x^2)}{\td x^{2\ell-1}}
&=\sum_{k=1}^{2\ell-1}\sin\Bigl(x^2+\frac{\pi}{2}k\Bigr) 2^k\bell_{2\ell-1,k}(x,1,\overbrace{0,\dotsc,0}^{2\ell-k-2})\\
&=\sum_{k=\ell}^{2\ell-1}2^k\sin\Bigl(x^2+\frac{\pi}{2}k\Bigr) \frac{a_{2\ell,2(k-\ell)+1}}{(2k-1)!!}{x^{2(k-\ell)+1}}.
\end{align*}
\end{enumerate}
\par
By the formulas in~\eqref{sin-cos-n-deriv}, if replacing the sine by the cosine in the above arguments, all results are also valid.
The proof of Theorem~\ref{sin-cos-deriv-thm} is complete.
\end{proof}

\begin{thm}\label{ln-deriv-thm}
For $\ell\in\mathbb{N}$, we have
\begin{equation}
\biggl(\ln\frac{1+x}{1-x}\biggr)^{(2\ell)}
=\frac{2x}{(1+x^2)^{\ell+1}}\sum_{k=0}^{\ell-1} \frac{[2(k+\ell)]!!}{[2(k+\ell)-1]!!} a_{2\ell,2k+1} \biggl(\frac{x^2}{1-x^2}\biggr)^k
\end{equation}
and
\begin{equation}
\biggl(\ln\frac{1+x}{1-x}\biggr)^{(2\ell-1)}
=\frac{2}{(1+x^2)^{\ell}}\sum_{k=0}^{2\ell-2} \frac{[2(k+\ell-1)]!!}{[2(k+\ell-1)-1]!!} a_{2\ell-1,2k} \biggl(\frac{x^2}{1-x^2}\biggr)^k,
\end{equation}
where $a_{n,k}$ are defined by~\eqref{recur-1}, \eqref{recur-5}, \eqref{arcsin-driv-coeffi-1}, and~\eqref{arcsin-driv-coeffi-n}.
\end{thm}

\begin{proof}
Let $u=u(x)=x^2$. Then, by Fa\'a di Bruno formula~\eqref{Bruno-Bell-Polynomial} and the formula~\eqref{Bell(n-n)}, we obtain
\begin{align*}
\biggl(\ln\frac{1+x}{1-x}\biggr)^{(n)}&=2\biggl(\frac1{1-x^2}\biggr)^{(n-1)}\\
&=2\sum_{k=1}^{n-1}\biggl(\frac1{1-u}\biggr)^{(k)}\bell_{n-1,k}(2x,2,\overbrace{0,\dotsc,0}^{n-k-2})\\
&=2\sum_{k=1}^{n-1}\frac{k!}{(1-u)^{k+1}}2^k\bell_{n-1,k}(x,1,\overbrace{0,\dotsc,0}^{n-k-2})\\
&=2\sum_{k=1}^{n-1}\frac{(2k)!!}{(1-x^2)^{k+1}}\bell_{n-1,k}(x,1,\overbrace{0,\dotsc,0}^{n-k-2}).
\end{align*}
Hence, by Theorem~\ref{Bell-exp-thm}, it follows that
\begin{enumerate}
\item
when $n=2\ell$, we have
\begin{align*}
\biggl(\ln\frac{1+x}{1-x}\biggr)^{(2\ell)}&=2\biggl(\frac1{1-x^2}\biggr)^{(2\ell-1)}\\
&=2\sum_{k=1}^{2\ell-1} \frac{(2k)!!}{(1-x^2)^{k+1}} \bell_{2\ell-1,k}(x,1,\overbrace{0,\dotsc,0}^{2\ell-k-2})\\
&=2\sum_{k=\ell}^{2\ell-1} \frac{(2k)!!a_{2\ell,2(k-\ell)+1}}{(2k-1)!!} \frac{x^{2k-2\ell+1}}{(1-x^2)^{k+1}};
\end{align*}
\item
when $n=2\ell-1$, we have
\begin{align*}
\biggl(\ln\frac{1+x}{1-x}\biggr)^{(2\ell-1)}&=2\biggl(\frac1{1-x^2}\biggr)^{(2\ell-2)}\\
&=2\sum_{k=1}^{2\ell-2} \frac{(2k)!!}{(1-x^2)^{k+1}} \bell_{2\ell-2,k}(2x,2,\overbrace{0,\dotsc,0}^{2\ell-k-3})\\
&=2\sum_{k=\ell-1}^{2\ell-2} \frac{(2k)!!a_{2\ell-1,2(k-\ell+1)}}{(2k-1)!!} \frac{x^{2(k-\ell+1)}}{(1-x^2)^{k+1}}.
\end{align*}
\end{enumerate}
The proof of Theorem~\ref{ln-deriv-thm} is complete.
\end{proof}

\begin{rem}
Since
\begin{equation*}
\biggl(\ln\frac{1+x}{1-x}\biggr)'=\frac2{1-x^2}=\frac{1}{x+1}-\frac{1}{x-1},
\end{equation*}
the $n$-th derivative of $\ln\frac{1+x}{1-x}$ may also be computed by
\begin{equation}
\biggl(\ln\frac{1+x}{1-x}\biggr)^{(n)}
=(-1)^{n-1}(n-1)!\biggl[\frac1{(x+1)^{n}}-\frac1{(x-1)^{n}}\biggr], \quad n\in\mathbb{N}.
\end{equation}
Similarly,
\begin{equation}
\frac{\td^n\ln(1-x^2)}{\td x^n}
=(-1)^{n-1}(n-1)!\biggl[\frac1{(x+1)^{n}}+\frac1{(x-1)^{n}}\biggr], \quad n\in\mathbb{N}.
\end{equation}
\end{rem}

\begin{thm}\label{ln(1+x-square)-thm}
For $\ell\in\mathbb{N}$, we have
\begin{equation}
\frac{\td^{2\ell}\ln(1+x^2)}{\td x^{2\ell}}
=\frac{(-1)^{\ell-1}2}{(1+x^2)^{\ell}}\sum_{k=0}^{\ell}(-1)^{k}\frac{[2(k+\ell-1)]!!}{[2(k+\ell)-1]!!} {a_{2\ell+1,2k}}\biggl(\frac{x^2}{1+x^2}\biggr)^k
\end{equation}
and
\begin{equation}
\frac{\td^{2\ell-1}\ln(1+x^2)}{\td x^{2\ell-1}}
=\frac{(-1)^{\ell-1}2x}{(1+x^2)^{\ell}}\sum_{k=0}^{\ell-1}(-1)^{k}\frac{[2(k+\ell-1)]!!}{[2(k+\ell)-1]!!} {a_{2\ell,2k+1}} \biggl(\frac{x^2}{1+x^2}\biggr)^k,
\end{equation}
where $a_{n,k}$ are defined by~\eqref{recur-1}, \eqref{recur-5}, \eqref{arcsin-driv-coeffi-1}, and~\eqref{arcsin-driv-coeffi-n}.
\end{thm}

\begin{proof}
Let $u=u(x)=x^2$. Using Fa\'a di Bruno formula~\eqref{Bruno-Bell-Polynomial} and the formula~\eqref{Bell(n-n)} yields
\begin{align*}
\frac{\td^n\ln(1+x^2)}{\td x^n}
&=\sum_{k=1}^n[\ln(1+u)]^{(k)}\bell_{n,k}(2x,2,\overbrace{0,\dotsc,0}^{n-k-1})\\
&=\sum_{k=1}^n(-1)^{k-1}\frac{(k-1)!}{(1+u)^k}2^k\bell_{n,k}(x,1,\overbrace{0,\dotsc,0}^{n-k-1})\\
&=2\sum_{k=1}^n(-1)^{k-1}\frac{(2k-2)!!}{(1+x^2)^k}\bell_{n,k}(x,1,\overbrace{0,\dotsc,0}^{n-k-1}).
\end{align*}
Consequently, by Theorem~\ref{Bell-exp-thm}, it follows that
\begin{enumerate}
\item
when $n=2\ell$, we have
\begin{align*}
\frac{\td^{2\ell}\ln(1+x^2)}{\td x^{2\ell}}
&=2\sum_{k=1}^{2\ell}(-1)^{k-1}\frac{(2k-2)!!}{(1+x^2)^k} \bell_{2\ell,k}(x,1,\overbrace{0,\dotsc,0}^{2\ell-k-1})\\
&=2\sum_{k=\ell}^{2\ell}(-1)^{k-1}\frac{(2k-2)!!}{(1+x^2)^k} \frac{a_{2\ell+1,2(k-\ell)}}{(2k-1)!!}{x^{2(k-\ell)}};
\end{align*}
\item
when $n=2\ell-1$, we have
\begin{align*}
\frac{\td^{2\ell-1}\ln(1+x^2)}{\td x^{2\ell-1}}
&=2\sum_{k=1}^{2\ell-1}(-1)^{k-1}\frac{(2k-2)!!}{(1+x^2)^k} \bell_{2\ell-1,k}(x,1,\overbrace{0,\dotsc,0}^{2\ell-k-2})\\
&=2\sum_{k=\ell}^{2\ell-1}(-1)^{k-1}\frac{(2k-2)!!}{(1+x^2)^k}  \frac{a_{2\ell,2(k-\ell)+1}}{(2k-1)!!}{x^{2(k-\ell)+1}}.
\end{align*}
\end{enumerate}
The proof of Theorem~\ref{ln(1+x-square)-thm} is complete.
\end{proof}

\begin{thm}\label{power-n-deriv-thm}
Let $\alpha\not\in\{0\}\cup\mathbb{N}$. For $\ell\in\mathbb{N}$, we have
\begin{equation}
\frac{\td^{2\ell}[(1\pm x^2)^\alpha]}{\td x^{2\ell}}
=\frac{(\pm2)^\ell}{(1\pm x^2)^{\ell-\alpha}}\sum_{k=0}^{\ell} \frac{(\pm2)^k\prod_{m=1}^{k+\ell}(\alpha-m+1)}{[2(k+\ell)-1]!!} a_{2\ell+1,2k} \biggl(\frac{x^2}{1\pm x^2}\biggr)^k
\end{equation}
and
\begin{equation}
\frac{\td^{2\ell-1}[(1\pm x^2)^\alpha]}{\td x^{2\ell-1}}
=\frac{(\pm2)^{\ell}x}{(1\pm x^2)^{\ell-\alpha}}\sum_{k=0}^{\ell-1}\frac{(\pm2)^{k} \prod_{m=1}^{k+\ell}(\alpha-m+1)}{[2(k+\ell)-1]!!} a_{2\ell,2k+1}\biggl(\frac{x^2}{1\pm x^2}\biggr)^k,
\end{equation}
where $a_{n,k}$ are defined by~\eqref{recur-1}, \eqref{recur-5}, \eqref{arcsin-driv-coeffi-1}, and~\eqref{arcsin-driv-coeffi-n}.
\end{thm}

\begin{proof}
Let $u=u(x)=\pm x^2$. Using Fa\'a di Bruno formula~\eqref{Bruno-Bell-Polynomial} and the formula~\eqref{Bell(n-n)} brings out
\begin{align*}
\frac{\td^n[(1\pm x^2)^\alpha]}{\td x^n}
&=\sum_{k=1}^n[(1+u)^\alpha]^{(k)}\bell_{n,k}(\pm2x,\pm2,\overbrace{0,\dotsc,0}^{n-k-1})\\
&=\sum_{k=1}^n \prod_{m=1}^{k}(\alpha-m+1)(1+u)^{\alpha-k} (\pm2)^k\bell_{n,k}(x,1,\overbrace{0,\dotsc,0}^{n-k-1})\\
&=\sum_{k=1}^n \Biggl[(\pm2)^k\prod_{m=1}^{k}(\alpha-m+1)\Biggr](1\pm x^2)^{\alpha-k} \bell_{n,k}(x,1,\overbrace{0,\dotsc,0}^{n-k-1}).
\end{align*}
As a result, by Theorem~\ref{Bell-exp-thm}, it follows that
\begin{enumerate}
\item
when $n=2\ell$, we have
\begin{align*}
\frac{\td^{2\ell}[(1\pm x^2)^\alpha]}{\td x^{2\ell}}
&=\sum_{k=1}^{2\ell}\Biggl[(\pm2)^k\prod_{m=1}^{k}(\alpha-m+1)\Biggr](1\pm x^2)^{\alpha-k} \bell_{2\ell,k}(x,1,\overbrace{0,\dotsc,0}^{2\ell-k-1})\\
&=\sum_{k=\ell}^{2\ell}\Biggl[(\pm2)^k\prod_{m=1}^{k}(\alpha-m+1)\Biggr](1\pm x^2)^{\alpha-k} \frac{a_{2\ell+1,2(k-\ell)}}{(2k-1)!!}{x^{2(k-\ell)}};
\end{align*}
\item
when $n=2\ell-1$, we have
\begin{align*}
\frac{\td^{2\ell-1}[(1\pm x^2)^\alpha]}{\td x^{2\ell-1}}
&=\sum_{k=1}^{2\ell-1}\Biggl[(\pm2)^k\prod_{m=1}^{k}(\alpha-m+1)\Biggr](1\pm x^2)^{\alpha-k} \bell_{2\ell-1,k}(x,1,\overbrace{0,\dotsc,0}^{2\ell-k-2})\\
&=\sum_{k=\ell}^{2\ell-1}\Biggl[(\pm2)^k\prod_{m=1}^{k}(\alpha-m+1)\Biggr](1\pm x^2)^{\alpha-k}  \frac{a_{2\ell,2(k-\ell)+1}}{(2k-1)!!}{x^{2(k-\ell)+1}}.
\end{align*}
\end{enumerate}
The proof of Theorem~\ref{power-n-deriv-thm} is complete.
\end{proof}

\begin{rem}
In general, the $n$-th derivatives of the function $h(x)=f(x^2)$ may be expressed as
\begin{equation}
h^{(2\ell)}(x)=\sum_{k=0}^{\ell} \frac1{[2(k+\ell)-1]!!}a_{2\ell+1,2k}{x^{2k}} f^{(k+\ell)}(x^2)
\end{equation}
and
\begin{equation}
h^{(2\ell-1)}(x)=\sum_{k=0}^{\ell-1} \frac1{[2(k+\ell)-1]!!} a_{2\ell,2k+1} {x^{2k+1}}f^{(k+\ell)}(x^2),
\end{equation}
where $\ell\in\mathbb{N}$ and $a_{n,k}$ are defined by~\eqref{recur-1}, \eqref{recur-5}, \eqref{arcsin-driv-coeffi-1}, and~\eqref{arcsin-driv-coeffi-n}.
\end{rem}

\section{Miscellanea}

By Fa\'a di Bruno formula~\eqref{Bruno-Bell-Polynomial}, we may establish
\begin{multline*}
-(\tan x)^{(n-1)}=(\ln\cos x)^{(n)}\\*
=\sum_{k=1}^n\frac{(-1)^{k-1}(k-1)!} {\cos^kx} \bell_{n,k}\Bigl(\cos\Bigl(x+\frac\pi2\Bigr), \dotsc,\cos\Bigl(x+(n-k+1)\frac\pi2\Bigr)\Bigr)
\end{multline*}
and
\begin{multline*}
(\cot x)^{(n-1)}=(\ln\sin x)^{(n)}\\*
=\sum_{k=1}^n\frac{(-1)^{k-1}(k-1)!} {\sin^kx} \bell_{n,k}\Bigl(\sin\Bigl(x+\frac\pi2\Bigr), \dotsc,\sin\Bigl(x+(n-k+1)\frac\pi2\Bigr)\Bigr).
\end{multline*}
It is possible that, by comparing and equating these derivatives with the formulas~\eqref{tan-deriv-unity} and~\eqref{cotan-deriv-n}, we may discover explicit expressions for Bell polynomials
\begin{equation*}
\bell_{n,k}\Bigl(\cos\Bigl(x+\frac\pi2\Bigr),\cos\Bigl(x+2\frac\pi2\Bigr), \dotsc,\cos\Bigl(x+(n-k+1)\frac\pi2\Bigr)\Bigr)
\end{equation*}
and
\begin{equation*}
\bell_{n,k}\Bigl(\sin\Bigl(x+\frac\pi2\Bigr),\sin\Bigl(x+2\frac\pi2\Bigr), \dotsc,\sin\Bigl(x+(n-k+1)\frac\pi2\Bigr)\Bigr).
\end{equation*}
These results may be applied to procure explicit formulas for the $n$-th derivatives of the functions $e^{\pm\sin x}$ and $e^{\pm\cos x}$.
\par
Utilizing Fa\'a di Bruno formula~\eqref{Bruno-Bell-Polynomial} and the formulas~\eqref{sin-cos-n-deriv} and~\eqref{Bell(n-n)}, we obtain
\begin{align*}
[\sin(e^{\pm x})]^{(n)}&=\sum_{k=1}^n\sin^{(k)}(e^{\pm x}) \bell_{n,k}\bigl((\pm1)e^{\pm x}, (\pm1)^2e^{\pm x},\dotsc, (\pm1)^{n-k+1}e^{\pm x}\bigr)\\
&=(\pm1)^n\sum_{k=1}^n\sin\Bigl(e^{\pm x}+\frac\pi2k\Bigr)e^{\pm kx} \bell_{n,k}(\overbrace{1,\dotsc, 1}^{n-k+1})\\
&=(\pm1)^n\sum_{k=1}^nS(n,k)\sin\Bigl(e^{\pm x}+\frac\pi2k\Bigr)e^{\pm kx}
\end{align*}
and
\begin{equation*}
[\cos(e^{\pm x})]^{(n)}=(\pm1)^n\sum_{k=1}^nS(n,k)\cos\Bigl(e^{\pm x}+\frac\pi2k\Bigr)e^{\pm kx},
\end{equation*}
where
\begin{equation}
\bell_{n,k}(\overbrace{1,\dotsc, 1}^{n-k+1})=S(n,k)
\end{equation}
may be found in~\cite[p.~135]{Comtet-Combinatorics-74} and
\begin{equation}\label{Stirling-Number-dfn}
S(n,k)=\frac1{k!}\sum_{\ell=0}^k(-1)^{k-\ell}\binom{k}{\ell}\ell^{n}
\end{equation}
is called Stirling number of the second kind which may be combinatorially interpreted as the number of partitions of the set $\{1,2,\dotsc,n\}$ into $k$ non-empty disjoint sets. For more information on Stirling numbers of the second kind $S(n,k)$, please refer to~\cite{Comtet-Combinatorics-74, BernPolRep.tex, Bernoulli-Stirling2-3P.tex, ANLY-D-12-1238.tex, exp-derivative-sum-Combined.tex, SirlingTo-Recurr-Ext.tex, Bell-Stirling-HyperGeom-JIS.tex, Bernoulli-No-Int-New.tex, Bell-Stirling-Lah-simp.tex, Eight-Identy-More.tex} and closely related references therein.

\begin{rem}
This paper is a slightly revised version of the preprint~\cite{Derivatives-arcsine-arccosine.tex}.
\end{rem}

\end{document}